\documentclass[a4paper]{amsart}

\usepackage{amssymb}
\usepackage{amsmath}
\usepackage{amsfonts}
\usepackage{amsthm}
\usepackage{enumitem}

\usepackage{lineno,hyperref}
\modulolinenumbers[5]

\newtheorem{thm}{Theorem}[section]
\newtheorem{corollary}[thm]{Corollary}
\newtheorem{fact}[thm]{Fact}
\newtheorem{example}[thm]{Example}
\newtheorem{lemma}[thm]{Lemma}
\newtheorem{proposition}[thm]{Proposition}
\newtheorem{remark}[thm]{Remark}

\numberwithin{equation}{section}
\newcommand{\N}{\mathbb{N}}
\newcommand{\Z}{\mathbb{Z}}

\newcommand{\R}{\mathbb{R}}

\begin{document}

\title{Infinite derivatives of the Takagi-Van der Waerden functions}

\date{\today\ (\the\time)}

\author{Juan Ferrera}
\address{IMI, Departamento de An\'alisis Matem\'atico y Matem\'atica Aplicada, Facultad de Ciencias Matem\'aticas, Universidad Complutense, 28040, Madrid, Spain}
\email{ferrera@mat.ucm.es}

\author{Javier G\'omez Gil}
\address{Departamento de An\'alisis Matem\'atico y Matem\'atica Aplicada, Facultad de Ciencias Matem\'aticas, Universidad Complutense, 28040, Madrid, Spain}
\email{gomezgil@mat.ucm.es}

\author{Jes\'us Llorente}
\address{Departamento de An\'alisis Matem\'atico y Matem\'atica Aplicada, Facultad de Ciencias Matem\'aticas, Universidad Complutense, 28040, Madrid, Spain}
\email{jesllore@ucm.es}


\begin{abstract}
In this paper we characterize the set of points where the lateral derivatives of the Takagi-Van der Waerden functions are infinite.
We also prove that the set  of points with infinite derivative has Hausdorff dimension one and Lebesgue measure zero. 
\end{abstract}

\thanks{The authors were partially supported by Grant
  MTM2015-65825-P}

\maketitle

\section{Introduction}

Takagi-Van der Waerden functions are an immediate generalization of the Takagi function (see \cite{T})
and they constitute a family of continuous nowhere differentiable functions. The first proof that we know of this fact can be found in \cite{BE}, other proofs can be found
in \cite{A} or in a more general setting in \cite{FGG}. The surveys \cite{AK} and \cite{L} are very good references on the Takagi function, its properties and generalizations.

For every integer $r\geq 2$, the Takagi-Van der Waerden function $f_r:[0,1]\to\mathbb{R}$ is defined as 
\begin{equation*}
f_r(x)=\sum_{n=0}^{\infty}\frac{1}{r^n}\phi(r^n x)
\end{equation*}
where $\phi(x)$ denotes the distance from the point $x$ to the nearest integer. 
Let us observe that  $f_2$ is the Takagi function, 
and $f_{10}$ is the Van der Waerden function (see \cite{Waerden}). 
This family of functions has been studied by many authors such as H. Whitney \cite{W}, 
J. B. Brown and G. Kozlowski \cite{BK}, A. Shidfar and K. Sabetfakhri \cite{SS}, Y. Baba \cite{B} or the authors we will mention hereunder.

The first two authors in \cite{FGG} introduce a generalization of the Takagi-Van der Waerden functions, 
named Generalized Takagi-Van der Waerden functions, which is defined on a separable real Hilbert space, and they also study 
the differentiation of the functions belonging to this generalization.

It is natural to ask about the set of points at which the Takagi-Van der Waerden functions 
possess a left-sided, right-sided or two-sided infinite derivative (see \cite{AK}). 
That is the issue we will try to address in this paper.  
Regarding this topic, M. Kr\"uppel \cite{KR1} and P. C. Allaart and K. Kawamura \cite{AK1} 
provide a complete characterization of the sets of points where the Takagi function possesses an infinite derivative. 
Also, P. C. Allaart and K. Kawamura \cite{AK1} prove that those sets of points have Hausdorff dimension one.

The aim of this paper is to fully characterize the set of points at which 
the Takagi-Van der Waerden function has an infinite derivative. 
In this sense, we will generalize the results obtained by M. Kr\"uppel \cite{KR1} and P. C. Allaart and K. Kawamura \cite{AK1}. 
Besides, we will prove that the set of points with infinite derivative has Hausdorff dimension one and Lebesgue measure zero.

Throughout this paper, we will write $f_r$ in terms of the corresponding 
Generalized Takagi-Van der Waerden function. 
For this reason, we briefly present this generalization defined on $[0,1]$.

Let $D$ be  a countable and dense subset of $[0,1]$. 
Let us consider $\mathcal{D}=\{ D_n\} _n$  an increasing sequence 
of finite subsets of $D$ satisfying that $D=\cup_n D_n$. 
We will say that $\mathcal{D}$ is a decomposition of $D$. 
Furthermore, we will denote the family of all connected components of 
$[0,1]\setminus D_n$ by $\mathcal{F}_n$ and by $\mathcal{F}$ the union of all the
families $\mathcal{F}_n$.

If $\mathcal{L}$ denotes the Lebesgue measure on $\R$, 
we will also require the following restriction on the decomposition $\mathcal{D}$: 
$$
 \mathcal{L}(I_n)\leq \alpha _n,
$$
for every $I_n\in \mathcal{F}_n$, where $\{ \alpha _n\}_n \in \ell ^1$ is a sequence of positive numbers satisfying the inequalities  
$2\alpha _{n+1}\leq \alpha _n$ for every $n$.    

For every decomposition $\mathcal{D}=\{ D_n\}$ of $D$, the Generalized Takagi-Van der Waerden function 
$T_{\mathcal{D}}:[0,1]\to \mathbb{R}$  is defined as
$$
T_{\mathcal{D}}(x)=\sum_{n=1}^{\infty} g_n(x)
$$
where $g_n(x)=\text{dist}(x,D_n)$ denotes the distance from $x$ to the set $D_n$. 
In \cite{FGG}, the following result was formulated  for a separable real Hilbert space when some restrictions on the decomposition $\mathcal{D}$ are required. However, in the case of $\mathbb{R}$ these restrictions are not necessary.

\begin{thm}\label{JOCA}
If $x\in D$, then $\partial T_{\mathcal{D}}(x)=\mathbb{R}$, otherwise $\partial T_{\mathcal{D}}(x)=\emptyset$.
\end{thm}

Here $\partial T_{\mathcal{D}}(x)$ denotes the Fr\'echet subdifferential of $T_{\mathcal{D}}$ at $x$ 
 (see \cite{F}). In the one dimensional case, the subdifferential of a function $f:\R\to\R$ at a point $x$ is characterized in terms of the Dini derivatives as follows
$$
\partial f(x)=\left [D^-f(x), d_+f(x)\right]\cap \mathbb{R}.
$$
Let us remember that the Dini derivatives are defined by
\begin{align*}
D^- f(x)=\limsup_{t\uparrow 0}\frac{f(x+t)-f(x)}{t}\quad \text{and}\quad d_+f(x)=\liminf_{t\downarrow 0}\frac{f(x+t)-f(x)}{t}.
\end{align*}
The other two Dini derivatives $D^+f(x)$ and $d_-f(x)$ are defined analogously.

If we consider the set $D=\{kr^{-n}\in[0,1]:k, n\in\mathbb{Z}^+\}$ and its decomposition
$$D_n=\left \{\frac{k}{r^{n-1}}\in [0,1]:k\in\mathbb{Z}\right \},$$
we may rewrite the function $f_r$ as the corresponding Generalized Takagi-Van der Waerden function given by
$$
f_r(x)=\sum_{n=1}^{\infty}g_n(x).
$$
Moreover, we will denote by $\widetilde{D}_n$ the set of middle points of consecutive points of $D_n$ 
and by $\widetilde{D}$ their union, $\cup_n \widetilde{D}_n$.

Let us describe the body of this paper. In Section 2 we will examine the behavior of the functions $g_n$,
as well as the derivatives series $\sum_{n=1}^{\infty}g'_n(x)$
when $x\not\in D\cup \widetilde{D}$. 
Let us observe that $g'_n(x)$ does not exist for $n$ big enough when $x\in D\cup \widetilde{D}$. 
Additionally, we will prove that the sets
$$
\Big\{ x\in [0,1]:\sum_{n=1}^{\infty}g'_n(x)=+\infty \Big\}
\quad \text{and}\quad 
\Big\{ x\in [0,1]:\sum_{n=1}^{\infty}g'_n(x)=-\infty \Big\}
$$
have Lebesgue measure zero.

At the beginning of the Section 3 we will characterize the lateral derivatives when $x\in D$ 
and when $x\in\widetilde{D}$. 
Secondly, we will investigate the relation of the derivatives series with lateral derivatives 
and Dini derivatives when $x\not\in D\cup \widetilde{D}$. 
It is important to highlight the role of $r$ parity when considering this relation. 
We will illustrate it with an example. 
Subsequently, we dedicate the rest of this section to present the results that completely characterize the set of points where the Takagi-Van der Waerden functions possess a left-sided or right-sided infinite derivative.

Finally, in Section 4 we will prove that the sets 
$$
\{ x\in [0,1]: f'_r(x)=+\infty \}\quad\text{and}\quad \{ x\in [0,1]: f'_r(x)=-\infty \}
$$
have Hausdorff dimension one and Lebesgue measure zero. 

\section{Behavior of the functions $g_n$  and their derivatives series}

We will begin by introducing some notation and basic results. For a real number $x\in [0,1]$ we consider its basis $r$ expansion
$$
x=\sum_{n=1}^{\infty}\frac{\varepsilon _n}{r^n}, \qquad \varepsilon _n \in \{ 0,1,\dots ,r-1\}.
$$

It is immediate that $x\in D$ if and only if there exists $n_0$ such that either 
$\varepsilon _n=0$ for every $n\geq n_0$ or $\varepsilon _n=r-1$ for every $n\geq n_0$. In this case,  unless expressly stated otherwise, we will choose the representation ending in all zeros. On the
other hand, regarding $\widetilde{D}$, we have to distinguish two cases: if $r$ is even then 
$\widetilde{D}\subset D$, but for $r$ odd, 
$x\in \widetilde{D}$ if and only if $\varepsilon _n=\frac{r-1}{2}$ for every $n\geq n_0$ for some $n_0$. 

It is clear that when $x\in D\cup \widetilde{D}$, $g'_n(x)$ does not exists for $n$ big enough, meanwhile
if $x\not\in D\cup \widetilde{D}$, then
the derivatives $g'_n(x)\in \{ -1,1\}$, are determined as follows:

\begin{enumerate}
  \item If $\varepsilon _n<\frac{r-1}{2}$, then $g'_n(x)=1$.
  \item If $\varepsilon _n>\frac{r-1}{2}$, then $g'_n(x)=-1$.
  \item If $\varepsilon _n=\frac{r-1}{2}$, then $g'_n(x)=g'_{k_n}(x)$ where
  $k_n=\min \left \{ j>n: \varepsilon _j\neq \frac{r-1}{2}\right \}$.
\end{enumerate}

Let us observe that the third case can only occur when r is odd. If $n\geq 2$ we denote by
$$
\hat{x}_n=\sum_{k=1}^{n-1}\frac{\varepsilon _k}{r^k},
$$
the biggest element of $D_n$ smaller than $x$, $\hat{x}_1=0$. It is immediate that $ g_n(x)=\min \left \{ x-\hat{x}_n,\hat{x}_n+\frac{1}{r^{n-1}}-x\right \}.$
More precisely, if $x\not\in D\cup \widetilde{D}$ we have

\begin{enumerate}
  \item If $\varepsilon _n<\frac{r-1}{2}$, then $g_n(x)=x-\hat{x}_n$.
  \item If $\varepsilon _n>\frac{r-1}{2}$, then $g_n(x)=\hat{x}_n+\frac{1}{r^{n-1}}-x$. 
  \item  If $\varepsilon _n=\frac{r-1}{2}$, then $g_{n}(x)=x-\hat{x}_n$ if and only if
  $\varepsilon _{k_n}<\frac{r-1}{2}$.
\end{enumerate}

The functions $g_n$, and consequently $f_r$ are symmetric with respect to $x=1/2$, hence
if we consider the function $S:[0,1]\to [0,1]$ defined by  $S(x)=1-x$,
it is immediate that $g_n(x)=g_n(S(x))$ and $g'_n(x)=-g'_n(S(x))$ for every $x\notin D\cup \widetilde{D}$. The following fact will be useful later on.

\begin{fact}\label{simetria}
For every $x\notin D\cup \widetilde{D}$ we have that
$\sum_{n=1}^{\infty}g'_n(x)=+\infty$ if and only if 
$\sum_{n=1}^{\infty}g'_n(S(x))=-\infty$. 
\end{fact}

We denote
$$
O_n= O_n(x):=\# \left \{ k\leq n: \varepsilon _k<\frac{r-1}{2}\right\} 
$$
and
$$
I_n=I_n(x):=\# \left \{ k\leq n: \varepsilon _k>\frac{r-1}{2}\right \} .
$$
Let us observe that $O_n+I_n\leq n$ with equality whenever $r$ is even. In this case $\frac{r-1}{2}\notin \Z$ and consequently $\varepsilon _n\neq \frac{r-1}{2}$ always.
However, if $r$ is odd, we may have $\varepsilon _n=\frac{r-1}{2}$ for some $n$. For this reason, we introduce
the function $\varphi _r:[0,1]\smallsetminus (D\cup \widetilde{D})\to [0,1]$ defined by
$$
\varphi _r(x)=\sum_{n=1}^{\infty}\frac{\widetilde{\varepsilon} _n}{r^n}
$$
where, if we denote, as above, by $k_n$ the smallest index greater than $n$ such that $\varepsilon _{k_n}\neq \frac{r-1}{2}$,
then
$\widetilde{\varepsilon}_n=\varepsilon _{k_n}$ whenever $\varepsilon _n= \frac{r-1}{2}$ and 
$\widetilde{\varepsilon}_n=\varepsilon _n$ otherwise. The following result is immediate:

\begin{proposition}
Assume that $x\notin D\cup \widetilde{D}$.
If $r$ odd, then
$g'_n(x)=g'_n(\varphi _r(x))$ for every $n$. Consequently
\begin{enumerate}
  \item  $\sum_{n=1}^{\infty}g'_n(x)=+\infty$ if and only if  $\lim_n(O_n(\varphi _r(x))-I_n(\varphi _r(x)))=+\infty$.
  \item $\sum_{n=1}^{\infty}g'_n(x)=-\infty$ if and only if  $\lim_n(O_n(\varphi _r(x))-I_n(\varphi _r(x)))=-\infty$.
\end{enumerate}
\end{proposition}

\begin{thm}\label{nullset}
The sets 
$$
A^+=\Big\{ x:\sum_{n=1}^{\infty}g'_n(x)=+\infty \Big\} \quad \textrm{and}
\quad A^-=\Big\{ x:\sum_{n=1}^{\infty}g'_n(x)=-\infty \Big\}
$$
have Lebesgue measure zero.
\end{thm}
\begin{proof}
These sets are obviously disjoint, and they have the same measure since,
by Corollary \ref{simetria}, 
$A^+=S\bigl( A^- \bigr)$.

On the other hand, if we denote $J^n_0=(0,\frac{1}{r^n})$ and 
$J^n_k=\frac{k}{r^n}+J_0$, we have that 
$$
J^n_k\cap A^+=\frac{k}{r^n}+\left (J^n_0\cap A^+\right ).
$$
Hence $\mathcal{L} (J^n_k\cap A^+)=\mathcal{L} (J^n_0\cap A^+)$ for every $k=0,\dots ,r^n-1$, which
implies 
$$
\mathcal{L} (J^n_k\cap A^+)=\frac{1}{r^n}\mathcal{L} (A^+)=\mathcal{L} (J^n_k)\mathcal{L} (A^+).
$$
Now, let $I\subset [0,1]$ be an arbitrary interval.  We have that $I$ is union of a countable amount of disjoint
intervals $J^n_k$ with $n\in \N$ and $k\in \{ 0,\dots ,r^n-1\}$ plus a null (countable) set. This implies that
$\mathcal{L} (I\cap A^+)=\mathcal{L} (I)\mathcal{L} (A^+)$ for every interval $I\subset [0,1]$. 
It is well known that a set enjoying that property measures either $0$ or $1$ necessarily, but it cannot
measure $1$ since $A^+$ and $A^-$ have the same measure. Therefore
$$
\mathcal{L} (A^-)=\mathcal{L} (A^+)=0.
$$
\end{proof}

\section{Characterization of infinite derivatives}

This section includes the main results of the paper. First we
characterize the lateral derivatives of $f_r$ when $x\in D$ and when $x\in\widetilde{D}$. The first result that we present in this section is an immediate consequence of Theorem \ref{JOCA}.

\begin{proposition}\label{prop1}
If $x\in D$, then $f'^{+}_r(x)=+\infty$ and $f'^{-}_r(x)=-\infty$.
\end{proposition}

Let us observe that if $r$ is even, then $\widetilde{D}_n\subset D_{n+1}$  and  therefore $\widetilde{D}\subset D$. However, if
$r$ is odd then $\widetilde{D}\cap D=\emptyset$. Nevertheless we have the following result:

\begin{proposition}
If $r$ is odd and $x\in \widetilde{D}$, then $f'^{+}_r(x)=-\infty$ and $f'^{-}_r(x)=+\infty$.
\end{proposition}
\begin{proof}
Let us consider the Generalized Takagi-Van der Waerden  function associated to the decomposition $\mathcal{D=}\{ \widetilde{D}_n\}$ of
$\widetilde{D}$ defined as
$$
T_{\mathcal{D}}(x)=\sum_{n=1}^{\infty}\widetilde{g}_k(x)
$$
where  $\widetilde{g}_k(x)=d(x,\widetilde{D}_k)$ denotes the distance of $x$ to the set $\widetilde{D}_k$. Let us observe that 
$$
\widetilde{g}_n(x)=\frac{1}{2}\frac{1}{r^{n-1}}-g_n(x)
$$
and consequently
$$
T_{\mathcal{D}}(x)=\frac{r}{2(r-1)}-f_r(x).
$$
From Theorem \ref{JOCA}, if we observe that $D^- T_{\mathcal{D}}(x)=-d_- f_r(x)$ and $d_+ T_{\mathcal{D}}(x)=-D^+f_r(x)$ we obtain the result.
\end{proof}

In the sequel we will consider the case $x\not\in D\cup \widetilde{D}$. 

\begin{lemma}
  \label{sec:char-infin-deriv}
  For all n if $x\in (a_n,b_n)\in\mathcal{F}_n$,  then
  \begin{equation*}\label{eq:1}
    \frac{f_r(b_n)-f_r(x)}{b_n-x}\leq \sum_{k=1}^ng'_k(x)\leq \frac{f_r(a_n)-f_r(x)}{a_n-x}.
  \end{equation*}

\end{lemma}

\begin{proof}
It is enough to observe that
\begin{multline*}
\frac{f_r(b_n)-f_r(x)}{b_n-x}=\sum_{k=1}^n\frac{g_k(b_n)-g_k(x)}{b_n-x}-\sum_{k=n+1}^{\infty}\frac{g_k(x)}{b_n-x}\\
\leq
\sum_{k=1}^n\frac{g_k(b_n)-g_k(x)}{b _n-x}\leq \sum_{k=1}^ng'_k(x)
\end{multline*}
since $g_k$ is linear on $[x,b_n]$ for $k\leq n$  provided that
$g'_k(x)=-1$. The other inequality is similar.
\end{proof}

Let $x\notin D\cup\widetilde{D}$. For every $n$, we denote
\begin{equation*}
    d_n=\min\{y-x: y\in D_n\cup\widetilde D_n, \, y>x\}.
\end{equation*}
\begin{remark}\label{sec:char-infin-deriv-2}
Let us observe that $0<d_n<\frac1{2r^{n-1}}$. Furthermore, we have that if $d_{n+1}<d_n$ then $d_n\geq \frac1{2r^n}$. 

We also observe that $d_n=d_{n+1}$, if and only if, one of the following situations occurs:
\begin{enumerate}
\item $\varepsilon_n=r-1$ and $\varepsilon_{k_{n+1}}>\frac{r-1}2$.
\item $\varepsilon_n=\frac{r-2}2$ and $\varepsilon_{n+1}\geq \frac r2$.
\item $\varepsilon_n=\frac{r-1}2$ and $\varepsilon_{k_{n+1}}<\frac{r-1}2$.
\end{enumerate}
where $k_n$ is defined as above. 

Moreover, $g_n$ is linear in $[x,x+d_n]$ for every $n$.
\end{remark}

There exist an strictly increasing sequence of  integers  $\{m_k\}_{k\geq 0}$, with $m_0=0$,
and a  strictly decreasing sequence  $\{x_k\}_{k\geq 1} \downarrow x$ such that:
\begin{equation*}
  \label{eq:8}
  x_k\in  D_n\cup\widetilde D_n, \quad \text{and}\quad x-x_k=d_n,\quad
  \text{for $m_{k-1}<n\leq m_k$.}
\end{equation*}
Let us observe that   $\frac1{2r^{m_k}}<d_{m_k}< \frac1{2r^{m_k-1}}$. 

\begin{lemma}
  \label{sec:char-infin-deriv-1}
  Let $x\not\in D\cup\widetilde D$.
    There exist  sequences $\{h_n\}\downarrow 0$
    and  $\{h'_n\}\downarrow 0$  such that:
    \begin{align*}
      \frac{f_r(x+h_n)-f_r(x)}{h_n}\geq & \sum_{k=1}^{n}g'_k(x)-\frac{r}{r-1}-1\\
 \frac{f_r(x-h'_n)-f_r(x)}{-h'_n}\leq                                        &\sum_{k=1}^{n}g'_k(x)+ \frac{r}{r-1}+1. 
      \end{align*}
\end{lemma}

\begin{proof}
We observe first that
   if $ j\geq  m_k$ and 
    $\frac1{2r^j}<h\leq  d_{m_k}$ then
  \begin{multline}\label{eq:4}
      \frac{f_r(x+h)-f_r(x)}{h}= \sum_{m=1}^{m_{k}}
      \frac{g_m(x+h)-g_m(x)}{h}+\sum_{m= m_{k}+1}^{\infty}
      \frac{g_m(x+h)-g_m(x)}{h} \\  \geq \sum_{m=1}^{m_{k}}
      g'_m(x)-\frac1{2h}\sum_{m= m_{k}+1}^{\infty}\frac1{r^{m-1}}\geq \sum_{m=1}^{m_{k}}
      g'_m(x) - r^{j-m_k}\frac{r}{r-1}.
  \end{multline}

  Let $n\geq 1$. If $n=m_k$ for some $k$ the
   result is immediate by \eqref{eq:4} if we take $h_n=d_{m_k}$.
If $m_k< n< m_{k+1}$ and $x_{k+1}\in D_{m_k+1}$, let
$h_n=d_{m_k}$. In this case $g'_m(x)=-1$ for
$m_k+2\leq m \leq m_{k+1}$ and therefore, by \eqref{eq:4},
\begin{equation*}
  \frac{f_r(x+h_n)-f_r(x)}{h_n}\geq \sum_{m=1}^{m_{k}}
    g'_m(x)-\frac{r}{r-1}\geq\sum_{m=1}^{n}
    g'_m(x)-1-\frac{r}{r-1}.
  \end{equation*} 
If  $m_k<n< m_{k+1}$ and
$x_{k+1}\in \widetilde D_{m_{k+1}}$, then we take $h_n=d_{m_{k+1}}$. Again by \eqref{eq:4}, we have that
\begin{equation*}
   \frac{f_r(x+h_n)-f_r(x)}{h_n} \geq  \sum_{m=1}^{m_{k+1}} g'_m(x) 
   -  \frac{r}{r-1}\geq \sum_{m=1}^{n} g'_m(x) 
   -   \frac{r}{r-1}
\end{equation*}
since $g'_m(x)=1$ for all $m_k<m\leq
m_{k+1}$. 

The second inequality is obtained by applying the previous one at  $1-x$ and using the symmetry of $f_r$. 

\end{proof}

The following proposition is an immediate consequence of
lemmas \ref{sec:char-infin-deriv} and \ref{sec:char-infin-deriv-1}.
\begin{proposition}\label{necesaria1}
  Let $x\not\in D\cup\widetilde D$.
  \begin{enumerate}
  \item  \label{item:4} If either $f'^+_r(x)=+\infty$ or $f'^-_r(x)=+\infty$, then $\sum_{k=1}^{\infty}g'_k(x)=+\infty$.
\item \label{necesaria2} If either $f'^+_r(x)=-\infty$ or $f'^-_r(x)=-\infty$, then 
$
\sum_{k=1}^{\infty}g'_k(x)=-\infty.
$
\item \label{reciprocoparcial} If $\sum_{k=1}^{\infty}g'_k(x)=+\infty$
  then $D^+f_r(x)=D^-f_r(x)=+\infty$.
\item \label{item:3}
  If $\sum_{k=1}^{\infty}g'_k(x)=-\infty$
then $d_+f_r(x)=d_-f_r(x)=-\infty$. 
 \end{enumerate}
\end{proposition}

From Theorem \ref{nullset},  Proposition \ref{necesaria1}~(\ref{item:4}) and Proposition \ref{necesaria1}~(\ref{necesaria2}), we deduce the next result.
%%%%%%%%%%%%%%%%%%%%%%%%%%%%%%

\begin{corollary}\label{corolariolaterales}
The sets $\{ x: f'^+_r(x)=+\infty \}$, $\{ x: f'^-_r(x)=+\infty \}$, $\{ x: f'^+_r(x)=-\infty \}$ 
and $\{ x: f'^-_r(x)=-\infty \}$ are nulls for the Lebesgue measure. 
\end{corollary}
 
As we will see in the example below, we cannot improve Proposition \ref{necesaria1}~(\ref{reciprocoparcial}) getting $f'^+_r(x)=+\infty$.
However, if $r$ is even we have the following result that extends
Theorem 3.1 in \cite{AK1}. 

\begin{proposition}
  \label{derechapar}
 If  $r$ is even, then 
    \begin{enumerate}
  \item \label{item:6}  $f'^+_r(x)=+\infty$ provided that
$ \sum_{k=1}^{\infty}g'_k(x)=+\infty$.
\item \label{izquierdarpar}
 $f'^-_r(x)=-\infty$ provided that $\sum_{k=1}^{\infty}g'_k(x)=-\infty$.
\end{enumerate}
\end{proposition}

\begin{proof}
 It suffices to prove the first statement since $f'^-_r(x)=-\infty$ if and only if
$f'^+_r(1-x)=+\infty$. If $d_{m_{k+1}}< h\leq d_{m_k}$ then
\begin{multline}\label{eq:6}
\frac{f_r(x+h)-f_r(x)}{h} \\ \geq  \sum_{m=1}^{m_k} g'_m(x) +
\sum_{m=m_k +1}^{m_{k+1}} \frac{g_m(x+h)-g_m(x)}{h}-
    \frac1{2h}  \sum_{m=m_{k+1}+1}^{\infty}\frac1{r^{m-1}}
    \\ =  \sum_{m=1}^{m_k} g'_m(x) +
    \sum_{m=m_k +1}^{m_{k+1}}
    \frac{g_m(x+h)-g_m(x)}{h}-\frac1{2hr^{m_{k+1}}}  \frac{r}{r-1}\\ \geq  \sum_{m=1}^{m_k} g'_m(x) +
    \sum_{m=m_k+1}^{m_{k+1}}
    \frac{g_m(x+h)-g_m(x)}{h}-
    \frac{r}{r-1}.
\end{multline}
since $\frac{1}{2r^{m_{k+1}}}<d_{m_{k+1}}$. If  $m_{k+1}>m_k+1$ then 
\begin{equation}\label{eq:12}
  \frac{f_r(x+h)-f_r(x)}{h}\geq \sum_{m=1}^{m_{k+1}}
    g'_m(x)-2-\frac{r}{r-1}.
  \end{equation}
  If $m_{k+1}=m_k+1$ \eqref{eq:12} is immediate from \eqref{eq:6}. 
\end{proof}

Observe that  it is not possible to have a full converse of Proposition \ref{necesaria1}  since, for $r=2$, in \cite{KR} it is
provided an
example of a point such that the series of the derivatives converges to $+\infty$ but the function
$f_2$ has not $+\infty$ derivative at that point.

Next example for $f_3$ was the first clue that we had of the importance of $r$ parity  while dealing with these properties. We omit the proof because the result will be an immediate consequence of a subsequent theorem.

\begin{example}\label{example}
Let 
$$
x=\sum_{n=1}^{\infty}\frac{\varepsilon _n}{3^n}, 
$$
where $\varepsilon _n=0$ if $n=10^k$ for some $k$,
and $\varepsilon _n=1$ otherwise. We have that $g'_n(x)=1$ for every $n$, but
$f'_3(x)\neq +\infty$. 
\end{example} 

Furthermore, in the previous example it is not difficult to see directly 
that $d_+f_3(x)=-\infty$ although $\sum_kg'_k(x)=+\infty$. Consequently, Proposition \ref{derechapar}, 
and Theorem 3.1 in \cite{AK1} do not hold for $f_3$ since
$\lim_n(O_n-I_n)=+\infty$ and $f'^+_3(x)\neq +\infty$, even if we define $O_n$ not
as $O_n(x)$ but as $O_n(\varphi _3(x))$. Finally, observe that  similar examples exist
for every $r$.

Now, we present Theorem \ref{main+} that characterizes the set of points where $f'^+_r(x)=+\infty$, and therefore  it extends Proposition \ref{derechapar}~(\ref{item:6}) and Proposition
\ref{necesaria1}~(\ref{item:4}) for all $r\geq 2$. On the other hand, 
with respect to conditions that guarantee that $f'^-_r(x)=+\infty$, Theorem \ref{main-} generalizes the results that appear in \cite{AK1} and \cite{KR}.

Let $r\geq 2$ and $x\not\in D\cup\widetilde{D}$, we arrange 
the infinite set $\{i:\varepsilon_i\neq \frac{r-1}{2}\}$ as an increasing sequence $\{i_n\}_n$. 
Observe that if $r$  is even, that set is $\mathbb{N}$.

\begin{lemma}
  \label{sec:char-infin-deriv-3}
Let $r\geq 2$ be an  integer and $\ell \geq 2$. Then, 
there exists $x_0\in (0,\ell)$ satisfying that $r^{-x_0}(\ell-x_0)+x_0<\log_r \ell+3$.
\end{lemma}

\begin{proof}
  We consider the function $\varphi(x)=r^{-x}(\ell-x)+x-\frac{\log \ell}{\log r}$ defined on $[0,\ell]$ and we observe that $\varphi(0)=\varphi(l)=l-\frac{\log \ell}{\log r}$. There exists $x_0\in (0,\ell)$ such that $(\ell - x_0)\log r +1=r^{x_0}$.
 Hence,
  \begin{align*}
    \varphi(x_0)\log r&=1-r^{-x_0} +\log \frac{r^{x_0}}{\ell}<1+\log(1+\log r)<1+\log r<3\log r. 
  \end{align*}
\end{proof}

\begin{thm}\label{main+}
$f'^{+}_r(x)=+\infty$ if and only if
$$
\lim_n \sum_{k=1}^{i_n}g'_k(x)- (i_{n+1}-i_n)+\log_r(i_{n+1}-i_n)=+\infty.
$$
\end{thm}
\begin{proof}
%%%%%%%%%%%%%%%%%%

We only have to prove the $r$ odd case, since for $r$ even, 
the condition reduces to the convergence to $+\infty$ of the series, and then the
result follows from Proposition \ref{derechapar} and Proposition
\ref{necesaria1}~(\ref{item:4}).
%%%%%%%%%%%%%%%%%%

``If'' part: Let
$\frac{1}{r^{p+1}}\leq
2h<\frac{1}{r^{p}}$ and let $n$ be such that
$i_n< p \leq i_{n+1}$. 
We denote
$$
\Delta_k(h)=\frac{g_k(x+h)-g_k(x)}{h}.
$$
For $k=1,2,\ldots$ we have that $\Delta _k(h)\in [-1,1]$. Now, we observe that if $m\geq p$ then
  \begin{equation}
    \label{eq:7}
    \left|\frac{f_r(x+h)-f_r(x)}h-\sum_{k=1}^m
   \Delta_k(h)\right| = \left|\sum_{k=m+1}^\infty\Delta_k(h)\right|\leq  \frac1{2h}\sum_{k=m+1}^{\infty}\frac1{r^{k-1}} \leq\frac{r^2}{r-1}.
  \end{equation}

  Let $1\leq k\leq i_n$.
  If
$\varepsilon_{i_n}\neq r-1$ then,  by Remark 
  \ref{sec:char-infin-deriv-2}, we have that $d_{i_n+1}< d_{i_n}$  and
  therefore $d_k\geq d_{i_n}\geq
  \frac1{2r^{i_n}}\geq\frac1{2r^{p-1}}>h$. Hence, 
  $g_k$ restricted to $[x,x+h]$ is linear for all $1\leq k\leq i_n$ and consequently $\Delta_k(h)=g'_k(x)$.

  If $\varepsilon_{i_n}= r-1$ then let $k_0=\max\{k<i_n:\;
\varepsilon_k<r-1\}$. In this case, again by Remark
\ref{sec:char-infin-deriv-2}, we have that $d_{k_0}>d_{k_0+1}$ and $d_{k_0}\geq \frac{1}{2r^{k_0}}\geq\frac1{2r^{i_n+1}}>h$. Therefore, $g_k$ restricted
to $[x,x+h]$ is linear for all $1\leq k\leq k_0 $ and then $\Delta
_k(h)=g_k'(x)$. As $g_k'(x)=-1$ if
$k_0<k\leq i_n$ then $\Delta
_k(h)\geq g_k'(x)$.

Thus,  in both cases we obtain that
$$ 
\sum_{k=1}^{i_n} \Delta _k(h)\geq \sum_{k=1}^{ i_n}g'_k(x).
$$
In particular, if $p=i_{n+1}= i_n+1$ we have, by \eqref{eq:7}, 
\begin{equation*}
    \sum_{k=1}^{\infty} \Delta _k(h)\geq  \sum_{k=1}^{i_n}g'_k(x) -2- \frac{r^2}{r-1}.
\end{equation*}
In what follows, we assume that $i_{n+1}- i_n>1$. 

If $\varepsilon_{i_{n+1}}>\frac{r-1}{2}$ then $g'_k(x)=-1$ for all $i_n<k\leq i_{n+1}$ and 
hence, by \eqref{eq:7}, we obtain 
\begin{equation*}
\sum_{k=1}^{\infty} \Delta _k(h) 
\geq   \sum_{k=1}^{i_{n+1}}g'_k(x)- \frac{r^2}{r-1}. 
\end{equation*}

If $\varepsilon_{i_{n+1}}<\frac{r-1}{2}$ then $x+d_k\in\widetilde{D}_k$ for all $i_n<k\leq i_{n+1}$ and, by Remark \ref{sec:char-infin-deriv-2}, we obtain that
$d_{i_n+1}=\dotsb=d_{i_{n+1}}<\frac1{2r^{i_{n+1}-1}}$.
If $h\leq d_{i_{n+1}}$ then
\begin{equation*}
  \sum_{k=1}^{\infty} \Delta _k(h) \geq                                       \sum_{k=1}^{i_{n+1}}g'_k(x)-\frac{r^2}{r-1}.
\end{equation*}
If $h>d_{i_{n+1}}$ we have that $x+d_k<x+h\leq x+\frac1{2r^p}$ and 
then
\begin{equation*}
g_k(x+h)-g_k(x)=x+d_k+\frac1{2r^{k-1}}-(x+h)-\left(x-\left(x+d_k-\frac1{2r^{k-1}}\right)\right)=2d_k-h,  
\end{equation*}
for every $k$ such that $i_n< k \leq  p$. By Remark \ref{sec:char-infin-deriv-2}, we have that $d_{i_{n+1}}>d_{i_{n+1}+1}$ and then $d_{i_{n+1}}\geq \frac{1}{2r^{i_{n+1}}}$. Therefore, 
\begin{multline}\label{muchasDesigualdades}
\sum_{k=i_n+1}^{p}\Delta _k(h)\geq (p-i_n)(-1+r^{p-i_{n+1}}) =-(p-i_n)+(p-i_n)r^{p-i_{n+1}}\\ =
-(i_{n+1}-i_n)-(p-i_{n+1})+(i_{n+1}-i_n)r^{p-i_{n+1}}-(i_{n+1}-p)r^{p-i_{n+1}}\\\geq -(i_{n+1}-i_n)+\log_r(i_{n+1}-i_n)-(i_{n+1}-p)r^{p-i_{n+1}}
 \\\geq
-(i_{n+1}-i_n)+\log_r (i_{n+1}-i_{n})-1.
\end{multline}
Finally, we conclude that 
\begin{equation*}
\sum_{k=1}^{\infty}\Delta _k(h)\geq \sum_{k=1}^{i_n}g'_k(x)-(i_{n+1}-i_n)+\log_r (i_{n+1}-i_{n})-1-\frac{r^2}{r-1}
\end{equation*}
which gives us the result.

For the converse we may assume that $i_{n+1}-i_n\geq 7$, since 
otherwise 
$$
\sum_{k=1}^{i_n}g'_k(x)-(i_{n+1}-i_n)+\log_r(i_{n+1}-i_n)\geq 
\sum_{k=1}^{i_n}g'_k(x)-7
$$
and the result follows from Proposition \ref{necesaria1}.

First, assume that $\varepsilon_{i_{n+1}}> \frac{r-1}2$. By Remark \ref{sec:char-infin-deriv-2}, we have that $d_{i_{n+1}-1}>d_{i_{n+1}}$ and then $d_{i_{n+1}-1}\geq \frac{1}{2r^{i_{n+1}-1}}$. Therefore, $\Delta_k\left(\frac{1}{2r^{i_{n+1}-1}}\right)=g'_k(x)$ for all $1\leq k \leq i_{n+1}-1$ and, by \eqref{eq:7}, we obtain  
\begin{equation*}
\sum_{k=1}^{\infty}
  \Delta_k\left(\frac{1}{2r^{i_{n+1}-1}}\right)-\frac{r^2}{r-1}\leq   \sum_{k=1}^{i_{n+1}} g'_k(x) .
\end{equation*}

On the other hand, if $\varepsilon_{i_{n+1}}< \frac{r-1}2$ and
$i_n<p\leq i_{n+1}-5$  then $d_{i_{n+1}}<d_{i_n}$ and,  as we have seen in the first part of the proof that, for
$\frac1{r^{p+1}}\leq 2h=\frac1{r^{p+\alpha}}<\frac1{r^p}$ we have
\begin{multline*}
  \sum_{k=1}^{p}\Delta
  _k(h)=\sum_{k=1}^{i_{n}}g'_k(x)+(p-i_n) \left(\frac{2d_{i_{n+1}}}{h}-1\right)\\
\leq \sum_{k=1}^{i_{n}}g'_k(x)+(p-i_n)\left(r^{p+\alpha-i_{n+1}+2}-1\right).
\end{multline*}
Let $x_0$ be as in Lemma \ref{sec:char-infin-deriv-3} for
$\ell=i_{n+1}-i_n-5$ and let
$p=i_{n+1}-5-[x_0]$. It is clear that $ i_n< p \leq i_{n+1}-5$. Let
$\alpha=[x_0]+1-x_0$ and $2h= r^{-(p+\alpha)}$. Then
\begin{equation*}
  (p-i_n+[x_0]-x_0)r^{p+\alpha-i_{n+1}+2}+x_0\leq\log_r(i_{n+1}-i_n-5)+3
\end{equation*}
and therefore

  \begin{multline*}
    (p-i_n)\left(r^{p+\alpha-i_{n+1}+2}-1\right) \leq 
    (p-i_n+[x_0]-x_0)\left(r^{p+\alpha-i_{n+1}+2}-1\right) \\
    \leq-(i_{n+1}-i_n)+\log_r(i_{n+1}-i_n)+8.
  \end{multline*}
Hence, by \eqref{eq:7}, we have
\begin{align*}
\frac{f_r(x+h)-f_r(x)}{h}&\leq \sum_{k=1}^{i_{n}}g'_k(x)-(i_{n+1}-i_n)+\log_r(i_{n+1}-i_n)+8+\frac{r^2}{r-1}.
\end{align*}
Letting $n$ to infinite, and therefore $h$ to $0^+$, we obtain
$$
\lim_n\left ( \sum_{k=1}^{i_n}g'_k(x)-(i_{n+1}-i_n)+\log _r(i_{n+1}-i_n)\right ) =+\infty.
$$
\end{proof}

\begin{remark}
In Example \ref{example},  $i_n=10^n$. Hence
\begin{multline*}
\lim_n\left( \sum_{k=1}^{i_n}g'_k(x)-(i_{n+1}-i_n)+\log _r(i_{n+1}-i_n)\right)\\
=\lim_n \left( 2i_n-i_{n+1}+ \log _r(i_{n+1}-i_n)\right) =-\infty .
\end{multline*}
Therefore $f'^+_3(x)\neq +\infty$. As a matter of fact it is not difficult to prove directly that
$d^+f_3(x)=-\infty$.
\end{remark}

Let $r\geq 2$ and $x\not\in D\cup\widetilde{D}$, we arrange the infinite set $\{ n:\varepsilon _n\neq 0\}$ as an increasing sequence 
$\{ n_k\}_k$.

\begin{thm}\label{main-}
$f'^-_r(x)=+\infty$ if and only if
$$
\lim_k\sum_{n=1}^{n_k}g'_n(x)-(n_{k+1}-n_k)+\log_r(n_{k+1}-n_k)=+\infty
$$
\end{thm}
\begin{proof}
``If" part: Let $\frac{1}{r^{p+1}}\leq 2h<\frac{1}{r^p}$, let $k$ such that $n_k\leq p< n_{k+1}$. We denote 
$$
\Delta _n(h)=\frac{g_n(x-h)-g_n(x)}{-h}.
$$
Now, as in Theorem \ref{main+}, we have that if $m\geq p$ then 
  \begin{equation}
    \label{eq:100}
     \left|\sum_{n=m+1}^\infty\Delta_n(h)\right| \leq\frac{r^2}{r-1}.
  \end{equation}

Let $1\leq n\leq  n_k$. If $g'_n(x)=1$ then $g_n(x)=x-\hat{x}_n$, and
$g_n(x-h)=x-h-\hat{x}_n$ since $h<\frac{1}{r^p}\leq \frac{1}{r^{n_k}}\leq x-\hat{x}_{n_k}\leq x-\hat{x}_n$. Consequently, $\Delta _n(h)=g'_n(x)=1$.
	Thus, we obtain that
$$
\sum_{n=1}^{n_k}\Delta_n(h)\geq \sum_{n=1}^{n_k}g'_n(x).
$$
In particular, if $n_{k+1}=n_k+1$ we have, by \eqref{eq:100}
$$
\sum_{n=1}^{\infty}\Delta_n(h)\geq \sum_{n=1}^{n_k}g'_n(x) -\frac{r^2}{r-1}.
$$
In what follows, we assume that $n_{k+1}- n_k>1$. 

On the other hand, let $n_k< n\leq p$. We have that $g'_n(x)=1$ and $\hat{x}_n=\hat{x}_p$. If $x-\hat{x}_p\leq h$ then $g_n(x-h)=\hat{x}_p-(x-h)$ and 
$$
\Delta _n(h)=\frac{\hat{x}_p-(x-h)-(x-\hat{x}_p)}{-h}=-1+\frac{2(x-\hat{x}_p)}{h},
$$
meanwhile if $x-\hat{x}_p> h$ then $g_n(x-h)=x-h-\hat{x}_p$ and $\Delta _n(h)=1$. In both cases, proceeding in a similar way as in \eqref{muchasDesigualdades}, we deduce
\begin{multline*}
\sum_{n=n_k+1}^{p}\Delta_n(h)\geq (p-n_k)\bigl( -1+r^{p-n_{k+1}}\bigr)\\ 
\geq
-(n_{k+1}-n_k)+\log _r(n_{k+1}-n_k)-1.
\end{multline*}
Finally, we conclude that
\begin{align*}
\sum_{n=1}^{\infty}\Delta _n(h)\geq 
\sum_{n=1}^{n_k}g'_n(x)-(n_{k+1}-n_k)+\log_r(n_{k+1}-n_k)-1-\frac{r^2}{r-1}
\end{align*}
which gives us the result.

For the converse we may assume that $n_{k+1}-n_k\geq 7$ since for the indices that not satisfy this inequality,
the necessary estimation follows from Proposition \ref{necesaria1}.

Let $n_k<p\leq n_{k+1}-5$ and $\frac{1}{r^{p+1}}\leq 2h=\frac{1}{r^{p+\alpha}}<\frac{1}{r^p}$. First, we observe that if $\Delta_n(h)=g'_n(x)$ for some $n$, then $\Delta_k(h)=g'_k(x)$ for all $1\leq k\leq n$. As we have seen in the first part of the proof, if $g'_{n_k}(x)=1$ then $\Delta_{n_k}(h)=g'_{n_k}(x)=1$. However, we can say more:

If $\varepsilon_{n_k}>\frac{r}{2}$ then $g'_{n_k}(x)=-1$ and $g_{n_k}(x-h)=\hat{x}_{n_k}+\frac{1}{r^{n_k -1}}-(x-h)$ since 
$$
h<\frac{1}{2r^p}\leq \frac{1}{2r^{n_k}}\leq \frac{\varepsilon_{n_k}}{r^{n_k}}-\frac{1}{2r^{n_k -1}}\leq x-\hat{x}_{n_k}-\frac{1}{2r^{n_k -1}}.
$$
Therefore, $\Delta_{n_k}(h)=g'_{n_k}(x)=-1$.

If $\varepsilon_{n_k}=\frac{r}{2}$, what happens only when $r$ is even, we distinguish the following cases: If $\varepsilon_{n_k -1}=\frac{r}{2}$ then $\Delta_{n_k-1}(h)=g'_{n_k-1}(x)=-1$ since 
$$
h<\frac{1}{2r^p}\leq \frac{1}{2r^{n_k}}\leq \frac{\varepsilon_{n_k}}{r^{n_k}}= \frac{\varepsilon_{n_k-1}}{r^{n_k-1}}+\frac{\varepsilon_{n_k}}{r^{n_k}}-\frac{1}{2r^{n_k -2}}\leq x-\hat{x}_{n_k-1}-\frac{1}{2r^{n_k -2}},
$$
meanwhile if $\varepsilon_{n_k -1}\neq\frac{r}{2}$ we are in one of the previous cases and then $\Delta_{n_k-1}(h)=g'_{n_k-1}(x)$.

Lastly, let us observe that if $\varepsilon_{n_k}=\frac{r-1}{2}$ then $g'_{n_k}(x)=1$ since $\varepsilon_{n_k +1}=0$. On the other hand, if $n_k<n\leq p$ then  $x-\hat{x}_n=x-\hat{x}_p\leq h$ and, as we have seen in the first part of the proof, we have that 
$$
\Delta _n(h)=-1+\frac{2(x-\hat{x}_p)}{h}.
$$
Finally, we conclude that 
\begin{multline*}
\sum_{n=1}^{p}\Delta
  _n(h)=\sum_{n=1}^{n_{k}-1}g'_n(x)+\Delta_{n_k}(h)+(p-n_k) \left(-1+\frac{2(x-\hat{x}_p)}{h}\right)\\
\leq \sum_{n=1}^{n_{k}}g'_n(x) + 2+(p-n_k)\left(r^{p+\alpha-n_{k+1}+3}-1\right),
\end{multline*}
and the result follows as in Theorem \ref{main+}.
\end{proof}

\begin{remark}
Let  $x$ be the point defined in Example \ref{example}. Then $f'^-_3(x)=+\infty$. 
\end{remark}

Theorems \ref{main-} and \ref{main+} allow us to deduce the following result for the derivatives $-\infty$.

\begin{corollary}
If we define $\{ p_n\}_n$ as the sequence on indices such that $\varepsilon _{p_n}\neq r-1$, then 
\begin{enumerate}
  \item $f'^+_r(x)=-\infty$ if and only if
   $$
  \lim_n\left( \sum_{k=1}^{p_n}g'_k(x)+(p_{n+1}-p_n)-\log _r(p_{n+1}-p_n)\right) =-\infty.
  $$
   \item $f'^-_r(x)=-\infty$ if and only if
    $$
  \lim_n\left( \sum_{k=1}^{i_n}g'_k(x)+(i_{n+1}-i_n)-\log _r(i_{n+1}-i_n)\right) =-\infty.
  $$
\end{enumerate}
\end{corollary}
\begin{proof}
It is enough to observe that $1-x=\sum_{k=1}^{\infty}\frac{(r-1)-\varepsilon_k}{r^k}$ and the result follows immediately from the previous theorems. 
\end{proof}

\section{Hausdorff dimension}

The aim of the results that we present in this section is to prove that the set of points that have  infinite  derivative has Hausdorff dimension one.  Let $n\geq 3$ be an odd integer. We define the set
$$
B^+_{n}=\left\{ z=\sum_{j=1}^n\frac{\varepsilon _j}{r^j}>0: 
O_n(\varphi _r(z))-I_n(\varphi _r(z))\geq 1, \varepsilon _{n}\neq \frac{r-1}{2}\right\}
$$
and we observe that $B^+_{n}\subset D_{n+1}\setminus \{0,1\}$. We will consider 
$$
x=\sum_{j=1}^{\infty}\frac{\varepsilon _j}{r^j},\qquad \varepsilon_j\in\{0,\ldots,r-1\}
$$
and $\overline{x}_m=\hat{x}_{m}+\varepsilon_mr^{-m}$.
On the other hand, we define the set
$$
A^+_{n,k}=\left \{ x: r^{nk}\overline{x}_{n(k+1)}-[r^{nk}\overline{x}_{n(k+1)}]\in B^+_n\right \}.
$$
Observe that $A^+_{n,k}$ is a finite union of closed intervals whose endpoints belong to $D_{n+1}$. In this case, we choose the representation ending in all $r-1$  for the right endpoint of each interval.

Concerning the Hausdorff dimension we have the following results. 

\begin{lemma}\label{lemimpar}
If $r$ is odd, then the Hausdorff dimension of the set 
$$
A^+_n=\bigcap_{k=0}^{\infty}A^+_{n,k}
$$
is greater than or equal to $\frac{(n-1)\log r-1}{n\log r}$.
\end{lemma}
\begin{proof}
Let us consider the $r^n$ intervals $J_d=[d,d+\frac{1}{r^n}]$ for $d\in D_{n+1}$ and let $\varphi _d:[0,1]\to J_d$ be the contractive mapping  defined
by  $\varphi _d(x)=d+\frac{x}{r^n}$. 
We observe that the family $\{ \varphi _d\} _{d\in B^+_n}$ satisfies the Moran's open set condition (see \cite{E} or \cite{Fa} for instance).
Let us remember that the symmetry function  $S$ is defined by $S(x)=1-x$. We have that 
$S(J_d)\cap J_d=\emptyset$ for every $d\in B^+_n$ since $\varepsilon _n\neq \frac{r-1}{2}$.
It is easy to see that for $d_1, d_2$ satisfying $S(J_{d_1})=J_{d_2}$,
we have that if $d_1\in B_n^+$ then $d_2\notin B^+_n$, meanwhile if
$d_1\notin B_n^+$ then either $d_2\in B^+_n$ or $d_2=0$. From these facts we deduce that
$$
\# B^+_n=\frac{r^{n-1}(r-1)}{2}-1.
$$
Finally, let us observe that $A_n^+$ is the unique 
non empty compact set that satisfies
$$
A_n^+=\bigcup_{d\in B^+_n }\varphi _d(A_n^+).
$$
Indeed, it is enough to realize that if $y\in A^+_{n,k}$ then $\varphi_d(y)\in A^+_{n,k+1}$, meanwhile if $x\in A_n^+$ then $y=r^n\left ( x-\overline{x}_{n}\right )\in A_n^+$ and $\varphi_{\overline{x}_{n}}(y)=x.$
Consequently, the Hausdorff dimension of $A_n^+$ is
$$
\frac{\log (\# B^+_n)}{n\log r}=\frac{\log (r^{n-1}\frac{r-1}{2}-1)}{n\log r}\geq\frac{(n-1)\log r-1}{n\log r}.
$$ 
\end{proof}
\begin{lemma}\label{lempar}
If $r$ is even, then the Hausdorff dimension of the set 
$$
A^+_n=\bigcap_{k=0}^{\infty}A^+_{n,k}
$$
is greater than or equal to $1-\frac{2+\log 2}{n\log r}$.
\begin{proof}
Using the same argument as in the previous lemma, it is enough to observe that when $r$ is even we have  
$$
\# B^+_n=\frac{r^{n}-2}{2}.
$$
\end{proof}
\end{lemma}
\begin{thm}
The set  
$\{ x:f'_r(x)=+\infty \}$ has Lebesgue measure zero and Hausdorff dimension one.
\end{thm}
\begin{proof}
It is immediate to see that the Lebesgue measure is zero from Corollary \ref{corolariolaterales}. Concerning the Hausdorff dimension, this is a consequence of the fact that 
$$
\left( A^+_n\smallsetminus (D\cup \widetilde{D})\right) \subset \left \{ x:f_r'(x)=+\infty \right \}
$$
for every $n$. Indeed, if $x\notin D\cup \widetilde{D}$ belongs to $A^+_n$ then 
$$
\sum_{j=1}^{\infty}g'_j(x)=\sum_{k=0}^{\infty}\sum_{j=kn+1}^{n(k+1)}g'_j(x)\geq
\sum_{k=0}^{\infty}1=+\infty.
$$
On the other hand, with the notation of Theorems \ref{main+} and \ref{main-},  if $r$ is odd we have that
$i_{m+1}-i_m\leq n$ and 
$n_{m+1}-n_m<2n$ for every $m$, meanwhile $i_{m+1}-i_m=1$ and $n_{m+1}-n_m<2n$ when $r$ is even. Finally, the result follows from Lemmas \ref{lemimpar} and \ref{lempar}.
\end{proof}

By defining the sets $B_n^-$, $A_{n,k}^-$ and $A_n^-$ in a similar way, the same arguments allow us to obtain the following result.
\begin{thm}
The set  
$\{ x:f'_r(x)=-\infty \}$ has Lebesgue measure zero and Hausdorff dimension one.
\end{thm}


\begin{thebibliography}{10}
\expandafter\ifx\csname href\endcsname\relax
  \def\href#1#2{#2} \def\path#1{#1}\fi

\bibitem{A}
P.~C. Allaart, {On the level
  sets of the {T}akagi--van der {W}aerden functions}, J. Math. Anal. Appl.
  419~(2) (2014) 1168--1180.
\newblock \href {http://dx.doi.org/10.1016/j.jmaa.2014.05.046}
{\path{doi:10.1016/j.jmaa.2014.05.046}}.

\bibitem{AK}
P.~C. Allaart, K.~Kawamura, The {T}akagi function: a survey, Real Anal.
Exchange 37~(1) (2011/12) 1--54.

\bibitem{AK1}
P.~C. Allaart, K.~Kawamura, The improper infinite derivatives of {T}akagi's
  nowhere-differentiable function, J. Math. Anal. Appl. 372~(2) (2010)
  656--665.
\newblock \href {http://dx.doi.org/10.1016/j.jmaa.2010.06.059}
  {\path{doi:10.1016/j.jmaa.2010.06.059}}.

\bibitem{B}
Y.~Baba, On maxima of {T}akagi-van der {W}aerden functions, Proc. Amer. Math.
  Soc. 91~(3) (1984) 373--376.
\newblock \href {http://dx.doi.org/10.2307/2045305}
{\path{doi:10.2307/2045305}}.

\bibitem{BE}
F.~A. Behrend, {Some remarks on
  the construction of continuous non-differentiable functions}, Proc. London
  Math. Soc. (2) 50 (1949) 463--481.
\newblock \href {http://dx.doi.org/10.1112/plms/s2-50.6.463}
  {\path{doi:10.1112/plms/s2-50.6.463}}.

\bibitem{BK}
J.~B. Brown, G.~Kozlowski, Smooth interpolation, {H}\"{o}lder continuity, and
  the {T}akagi--van der {W}aerden function, Amer. Math. Monthly 110~(2) (2003)
  142--147.
\newblock \href {http://dx.doi.org/10.2307/3647774}
  {\path{doi:10.2307/3647774}}.

\bibitem{E}
G.~Edgar, {Measure,
  Topology, and Fractal Geometry}, Undergraduate Texts in Mathematics, Springer
  New York, 2007.

\bibitem{Fa}
K.~Falconer, {Fractal
  Geometry: Mathematical Foundations and Applications}, Wiley, 2004.

\bibitem{F}
J.~Ferrera, An introduction to nonsmooth analysis, Elsevier/Academic Press,
  Amsterdam, 2014.
\newblock \href {http://dx.doi.org/10.1016/B978-0-12-800731-0.00001-1}
  {\path{doi:10.1016/B978-0-12-800731-0.00001-1}}.

\bibitem{FGG}
J.~Ferrera, J.~G\'{o}mez~Gil, Generalized {T}akagi--{V}an der {W}aerden
  functions and their subdifferentials, J. Convex Anal. 25~(4) (2018)
  1355--1369.

\bibitem{KR}
M.~Kr\"{u}ppel, On the extrema and the improper derivatives of {T}akagi's
  continuous nowhere differentiable function, Rostock. Math. Kolloq.~(62)
  (2007) 41--59.

  \bibitem{KR1}
M.~Kr\"{u}ppel, On the improper derivatives of {T}agaki's continuous nowhere
  differentiable function, Rostock. Math. Kolloq.~(65) (2010) 3--13.
  
  \bibitem{L}
J.~C. Lagarias, The {T}akagi function and its properties, in: Functions in
  number theory and their probabilistic aspects, RIMS K\^{o}ky\^{u}roku
  Bessatsu, B34, Res. Inst. Math. Sci. (RIMS), Kyoto, 2012, pp. 153--189.

\bibitem{SS}
A.~Shidfar, K.~Sabetfakhri, On the {H}\"{o}lder continuity of certain
functions, Exposition. Math. 8~(4) (1990) 365--369.

\bibitem{T}
T.~Takagi, A simple example of the continuous function without derivative,
  Proc. Phys. Math. Soc. Tokio Ser. II~(1) (1903) 176--177.

\bibitem{Waerden}
B.~L. van~der Waerden, Ein einfaches {B}eispiel einer nicht-differenzierbaren
  stetigen {F}unktion, Math. Z. 32~(1) (1930) 474--475.
\newblock \href {http://dx.doi.org/10.1007/BF01194647}
  {\path{doi:10.1007/BF01194647}}.

\bibitem{W}
H.~Whitney, On totally differentiable and smooth functions, Pacific J. Math. 1
  (1951) 143--159.

\end{thebibliography}
\end{document}